\documentclass[12 pt]{article}
\def\qed{\quad{$\square$}} 
\def\db#1{{ \color{blue}#1}}
\usepackage{bm}
\usepackage{amsmath}
\usepackage{graphicx}
\usepackage{amssymb}
\usepackage{epsfig}
\usepackage{float}
\usepackage{mathrsfs}
\usepackage{caption}
\usepackage{subcaption}
\usepackage{dblfloatfix}
\usepackage{booktabs}

\usepackage{color}

\usepackage{lipsum}
\usepackage{tikz}
\newtheorem{theorem}{Theorem}
\newtheorem{corollary}{Corollary}

\newtheorem{assumption}{Assumption}{\bf}{}
\newtheorem{remark}{Remark}
\newtheorem{example}{Example}

\title{\LARGE \bf Nonlinear network dynamics for interconnected micro-grids%\thanks{Research supported by  the PRIN 20103S5RN3 ``Robust decision making in markets and organization". A short version of this paper including only the deterministic case will appear in a chapter of the book ``Game Theory with Engineering Applications'', which will be published by SIAM in April 2016. The Stochastic case in Section 4.2 and the model mis-specification in Section 4.3 are additional contributions. }
}

\author{Dario Bauso\thanks{Dario Bauso is with the Department of Automatic Control and Systems Engineering, The University of Sheffield, Mappin Street  Sheffield, S1 3JD, United Kingdom, and with the Dipartimento di Ingegneria Chimica, Gestionale, Informatica, Meccanica,  Universit\`a di Palermo, V.le delle Scienze, 90128 Palermo, Italy.   {\tt\small \{d.bauso@sheffield.ac.uk\}}}
}

\begin{document}

\allowdisplaybreaks

\maketitle
\thispagestyle{empty}
\pagestyle{empty}

%%%%%%%%%%%%%%%%%%%%%%%%%%%%%%%%%%%%%%%%%%%%%%%%%%%%%%%%%%%%%%%%%%%%%%%%%%%%%%%%
\begin{abstract}
This paper deals with transient stability in interconnected  micro-grids.
The main contribution involves i) robust classification of transient dynamics for different intervals of the micro-grid parameters (synchronization, inertia, and damping); ii) exploration of the analogies with consensus dynamics and bounds on the damping coefficient separating underdamped and overdamped dynamics iii) the extension to the case of disturbed measurements due to hackering or parameter uncertainties.  %Theoretical findings are illustrated on a case study on the Nigerian grid. 
 \end{abstract}

%\begin{keywords} 
\textbf{Keywords:} Synchronization; consensus; Nonlinear control; Transient stability. %\end{keywords}
% \HISTORY{This paper was
%first submitted on July 1, 2013. %April 12, 1922 and has been with the authors for
%83 years for 65 revisions.
%}

%\maketitle
%%%%%%%%%%%%%%%%%%%%%%%%%%%%%%%%%%%%%%%%%%%%%%%%%%%%%%%%%%%%%%%%%%%%%%

% Samples of sectioning (and labeling) in OPRE
% NOTE: (1) \section and \subsection do NOT end with a period
%       (2) \subsubsection and lower need end punctuation
%       (3) capitalization is as shown (title style).
%
%\section{Introduction.}\label{intro} %%1.
%\subsection{Duality and the Classical EOQ Problem.}\label{class-EOQ} %% 1.1.
%\subsection{Outline.}\label{outline1} %% 1.2.
%\subsubsection{Cyclic Schedules for the General Deterministic SMDP.}
%  \label{cyclic-schedules} %% 1.2.1
%\section{Problem Description.}\label{problemdescription} %% 2.

% Text of your paper here

\section{Introduction}
This paper investigates transient stability of interconnected micro-grids. 
First we develop a model for a single micro-grid  combining  swing dynamics and synchronization, inertia and damping parameters. We focus on the main characteristics of the transient dynamics especially the insurgence of oscillations in underdamped transients. The analysis of the transient dynamics is then extended to multiple interconnected micro-grids.  By doing this we  relate the transient characteristics to the connectivity of the graph. We also investigate  the impact of the disturbed measurements (due to hackering or parameter uncertainties) on the transient.

%\noindent
%\textbf{Main theoretical findings.} 

\subsection{Main theoretical findings}
The contribution of this paper is three-fold. First, for the single micro-grid we identify intervals for the parameters within which the behavior of the transient stability has similar characteristics. This shows robustness of the results and extends the analysis  to cases where the inertia, damping and synchronization parameters are uncertain. In particular we prove that underdamped dynamics and oscillations arise  when the damping coefficient is below a certain threshold which we calculate explicitly. The threshold is obtained as function of the product between the inertia coefficient and the synchronization parameter. 

Second, for interconnected micro-grids, under the hypothesis of homogeneity, we prove that the transient stability mimics a consensus dynamics and provide bounds on the damping coefficient for the consensus value to be overdamped or underdamped. This result is meaningful as it sheds light on  the insurgence of topology-induced oscillations.  These bounds depend on the topology of the grid and in particular on its maximum connectivity, namely, the maximum number of links over all the nodes of the network.  We also observe that the consensus value changes dramatically with increasing damping coefficient. This implies that the micro-grid, if working in islanding mode,  can synchronize to a frequency which deviates from the nominal one of 50 Hz. This finding extends to smart-grids with different inertia but  same ratio between damping and inertia coefficient. 

% We specialize the findings to the case of two interconnected micro-grids. 

Third,  we extend the analysis to the case where both frequency and power flow measurements are subject to disturbances. Using a traditional technique in nonlinear analysis and control we isolate the nonlinearities in the feedback loop, and analyze stability under some mild assumptions on the nonlinear parameters.  The obtained result extends also to the case where the model parameters like synchronization coefficient, inertia and damping coefficients are uncertain.  This adds robustness to our findings and proves validity of the results even under modeling errors. 

To corroborate our theoretical findings  a case study from the Nigerian distribution network is discussed.

%\noindent
%\textbf{Related literature}

\subsection{Related literature}

This study leverages on previous contributions of the authors in \cite{BagBau13} and  \cite{B17}. In \cite{BagBau13} the author  studies flexible demand in terms of a population of smart thermostatically controlled loads and shows that the transient dynamics can be accommodated within the mean-field game theory. 
In \cite{B17} the author extends the analysis to uncertain models involving both stochastic and deterministic (worst-case) analysis approaches. The analysis of interconnected micro-grids builds on previous studies provided in 
\cite{DB12}. Here the authors link transient stability in multiple  electrical generators to synchronization in a set of coupled Kuramoto oscillators. The connection between Kuramoto oscillators and consensus dynamics is addressed in~\cite{OSFM07}. A game perspective on Kuramoto oscillators is in  \cite{YMMS12}, where it is shown that the synchronization dynamics admits an interpretation as game dynamics with equilibrium points corresponding to Nash equilibria. The observed deviation of the consensus value from the nominal mains frequency in the case of highly overdamped dynamics can be linked to inefficiency of equilibria as discussed in \cite{YMMS14-DGAA}.
 This study has benefited from some graph theory tools and analysis efficiently and concisely exposed in \cite{lns-v.95}. 
The model used in this paper, which combines swing dynamics with synchronization, inertia and damping parameters has been inspired by \cite{MKC13}. The numerical analysis has been conducted using data provided in~\cite{AO13}.

\smallskip 

This paper is organized as follows. In Section \ref{sec:smc}, we model a single micro-grid.  In Section \ref{sec:mmc}, we turn to multiple interconnected micro-grids.  
In Section \ref{sec:as}, we analyze the impact of measurement disturbances. In Section \ref{sec:sim}, we provide numerical studies on the Nigerian grid. Finally, in Section \ref{sec:conc}, we provide conclusions.

%\newpage

%\newpage

\section{Model of a single micro-grid}\label{sec:smc}
%\subsection{Microscopic model - smart grid $i$}
Consider a single micro-grid connected to the network, refer to it as the $i$th micro-grid. 
Let us denote by $P_i$ the power flow into the $i$th micro-grid. Also let $f_i$ be the frequency deviation of micro-grid $i$ and $f_j$ a virtual signal representing the frequency of the mains. By applying dc approximation,  the power $P_i$ evolves according to 
\begin{equation}\label{dynamics}
%\left\{
\begin{array}{lll}
\dot P_i = T_{ij} (f_j-f_i)=T_{ij} e_{ij},%\\ \\
%\dot f_i = - \frac{D_i}{M_i} f_i + \frac{P_i}{M_i}
\end{array}%\right.
\end{equation}
where $T_{ij}$ is the synchronizing coefficient. This coefficient is obtained as the inverse of the transmission reactance between micro-grid $i$ and $j$. 
In other words, the  power $P_i$ depends on the frequency error $e_{ij}=f_j-f_i$. The physical intuition of this is that in response to a positive error we have power  injected into the $i$th micro-grid  from the $j$th micro-grid. Vice versa, a negative error induces power from micro-grid $i$ to $j$.

The dynamics for $f_i$ follows a traditional swing equation 
\begin{equation}\label{dynamics1}
%\left\{
\begin{array}{lll}
%\dot P_i = T_{ij} (f_j-f_i)=T_{ij} e_{ij},\\ \\
\dot f_i = - \frac{D_i}{M_i} f_i + \frac{P_i}{M_i},
\end{array}%\right.
\end{equation}
where $M_i$ and $D_i$ are the inertia and damping constants of the $i$th micro-grid, respectively.
By denoting 
$f_i=x^{(i)}_1$, $P_i=x^{(i)}_2$, $f_j=x^{(j)}_1$, and by considering $f_j$ as an exogenous input to the $i$th micro-grid, the dynamics of the $i$th micro-grid reduces to the following second-order system
\begin{equation}\label{dynamicsm}
\left[
\begin{array}{c}
\dot x^{(i)}_1\\
\dot x^{(i)}_2
\end{array}
\right]
 = \left[
\begin{array}{cc}
-\frac{D_i}{M_i} & \frac{1}{M_i}\\
-T_{ij} & 0
\end{array}
\right]\left[
\begin{array}{c}
 x^{(i)}_1\\
 x^{(i)}_2
\end{array}
\right] + 
\left[
\begin{array}{c}
0\\
T_{ij}
\end{array}
\right] x_i^{(j)}.
\end{equation}

Figure \ref{bs1} shows the block representation and corresponding transfer function of the dynamical  system (\ref{dynamicsm}).

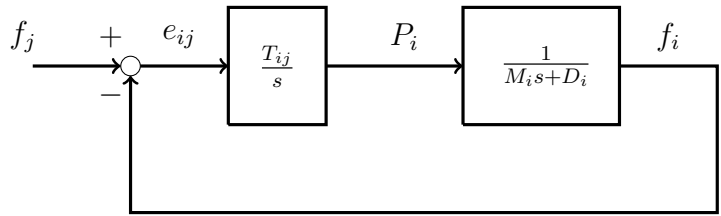
\begin{figure} [h]
\centering
\def\svgwidth{.8\columnwidth}
\begin{tikzpicture}
\begin{scope}[shift={(0,0)},scale=1.3]
\draw [->,very thick] (-1,0) -- (-.1,0);
\draw (0,0) circle (.1cm);
\draw [->,very thick] (.1,0) -- (1,0);
\draw [very thick] (1,-.6) -- (1,.6) -- (2,.6) -- (2,-.6) -- (1,-.6);
\draw [->,very thick] (2,0) -- (3.4,0);
\draw [very thick] (3.4,-.6) -- (3.4,.6) -- (5,.6) -- (5,-.6) -- (3.4,-.6);
\draw [->,very thick] (5,0) -- (6,0) -- (6,-1.5) --  (3.6,-1.5) -- (0,-1.5) -- (0,-.1);
%\draw [very thick] (2.0,-.6-2) -- (2.0,.6-2) -- (3.6,.6-2) -- (3.6,-.6-2) -- (2.0,-.6-2);
%\draw [->,very thick] (3.6,-2) -- (0,-2) -- (0,-.1) ;
\node at (-1.1,0.3) {$f_j$};\node at (-0.2,0.3) {$+$};
\node at (-0.2,-0.3) {$-$};\node at (0.5,0.3) {$e_{ij}$};
\node at (1.5,0.0) {$\frac{T_{ij}}{s}$};
\node at (4.25,0.0) {$\frac{1}{M_i s + D_i}$};
\node at (5.5,0.3) {$f_i$};
%\node at (5.5,0.3) {$f_i$};
%\node at (2.7,-2) {$\psi(f_i,t)$};
\node at (2.8,.3) {$P_i$};
%\node at (2.99,.4) {$\omega$};
%\draw (2.8,0.0) circle (.1cm);
%\draw [->,very thick] (2.9,0) -- (3.4,0);
%\draw [->,very thick] (2.8,.7) -- (2.8,0.1);
%\draw [->,very thick] (2.4,0) -- (2.4,-1.4);
\end{scope}
      \end{tikzpicture}\caption{Block representation of the $i$th micro-grid.}\label{bs1}
\end{figure}

\begin{theorem}\label{thm1}
Dynamics (\ref{dynamicsm}) is asymptotically stable. Furthermore, let $D_i>2 \sqrt{T_{ij} M_i}$ then the origin is an asymptotically stable node. Vice versa, if  $D_i<2 \sqrt{T_{ij} M_i}$ then the origin is an asymptotically stable spiral.
\end{theorem}
\begin{proof}
 For the first part, stability derives from $Tr(A)=-\frac{D_i}{M_i}$, where $Tr(A)$ is the trace of matrix $A$ and from $\Delta(A)=\frac{T_{ij}}{M_i}>0$, where $\Delta(A)$ is the determinant of matrix $A$. Let us recall that stability depends on the eigenvalues of $A$ and that the expression of the  eigenvalues is given by
\begin{equation}\label{eigss}
\begin{array}{lll}
\lambda_{1,2} & = \frac{Tr(A) \pm \sqrt{Tr(A)^2 - 4 \Delta(A)}}{2} \\
 & = \frac{1}{2} \Big(-\frac{D_i}{M_i} \pm \sqrt{(\frac{D_i}{M_i})^2 - 4 \frac{T_{ij}}{M_i}}\Big).
\end{array}
\end{equation} 
As the trace $Tr(A)$ is strictly negative and the determinant $ \Delta(A)$ is strictly positive, then this corresponds to any point in the fourth quadrant in Fig \ref{eigs}, which characterizes stable systems.

As for the rest of the proof,  we know that if $D_i>2 \sqrt{T_{ij} M_i}$ then $Tr(A)^2> 4 \Delta(A)$ and the origin is an asymptotically stable node. This corresponds to any point in the fourth quadrant in Fig \ref{eigs} outside the parabolic curve, whereby the system is stable and no oscillations occur. The parabolic curve identifies the set of points for which $Tr(A)^2=4 \Delta(A)$.

The last case is when  $D_i<2 \sqrt{T_{ij} M_i}$ which implies $Tr(A)^2<4 \Delta(A)$ and therefore the origin is an asymptotically stable spiral. This corresponds to any point in the fourth quadrant in Fig \ref{eigs}, inside the parabolic curve whereby the system is stable but oscillations may occur due to imaginary parts in the eigenvalues.
\end{proof}

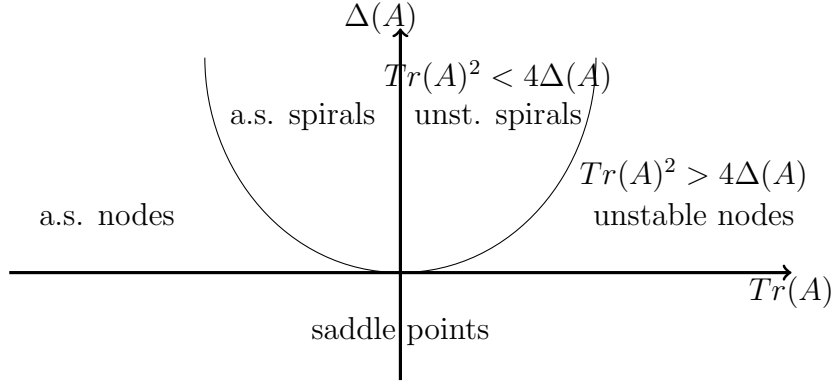
\begin{figure} [t]
\centering
\def\svgwidth{.8\columnwidth}
\begin{tikzpicture}
\begin{scope}[shift={(0,0)},scale=1.3,rotate=0]
 \draw [->,very thick] (-4,1.5) -- (4,1.5);
  \draw [->,very thick] (0,.4) -- (0,4);
  \draw (0,1.5) to[out=0,in=-90] (2.0,3.7);
  \draw (0,1.5) to[out=180,in=-90] (-2.0,3.7);
  \node at (1.0,3.5) {$Tr(A)^2<4\Delta(A)$};
  \node at (-.0,0.9) {saddle points};
    \node at (3,2.5) {$Tr(A)^2>4\Delta(A)$};
    \node at (3,2.1) {unstable nodes};
 \node at (-3,2.1) {a.s. nodes};
\node at (-1,3.1) {a.s. spirals};
\node at (1,3.1) {unst. spirals};
     \node at (4,1.3) {$Tr(A)$};
     \node at (-0.2,4.1) {$\Delta(A)$};
     \end{scope}
      \end{tikzpicture} 
      \caption{Classification of equilibrium points.}\label{eigs}
\end{figure}

The above theorem sheds light on the role of the different parameters in the transient stability of the micro-grid. 
\begin{example} In particular, let the synchronization coefficient be $T_{ij}=1$ and the inertia coefficient be $M=1$ and investigate the role of the damping coefficient $D$.  
From (\ref{eigss}) the transient dynamics is determined by the eigenvalues
 $\lambda_{1,2}= \frac{-D_i \pm \sqrt{D_i^2 - 4}}{2}.$ 
We can conclude that 

\begin{itemize}
\item if $D > 2$ all eigenvalues are real and negative and no oscillations arise. The slowest eigenmode is determined by the smallest (in modulus) eigenvalue, which is $\frac{-D_i + \sqrt{D_i^2 - 4}}{2}$. 

\item Differently, if $D\leq 2$ we have complex eigenvalues given by $\lambda_{1,2}= - \frac{D_i}{2} \pm i \frac{\sqrt{D_i^2 - 4}}{2}$ and we observe damped oscillations. The damping factor depends on the real part $ Re(\lambda_{1,2})=- \frac{D_i}{2}$  while oscillation frequencies are related to the imaginary part $Im(\lambda_{1,2})= \frac{\sqrt{D_i^2 - 4}}{2}$.
\end{itemize}
\end{example}
\begin{example} In this example we set the damping coefficient $D=1$  and the inertia coefficient  $M=1$ and investigate the role of the synchronization coefficient $T_{ij}$.   Again, from (\ref{eigss}), the eigenvalues governing the transient dynamics are
 $\lambda_{1,2}= \frac{-1 \pm \sqrt{1 - 4 T_{ij}}}{2}.$ 
Then we have the following cases: 

\begin{itemize}
\item if $T_{ij} < \frac{1}{4}$ the eigenvalues are all  real and negative and we observe no oscillations. The transient is dominated by the slowest eigenmode, which in turn is determined by the smallest (in modulus) eigenvalue, i.e. $\frac{-1+ \sqrt{1 - 4T_{ij}}}{2}$. 

\item Unlikewise, if $T_{ij} > \frac{1}{4}$ the eigenvalues are complex and given by $\lambda_{1,2}= - \frac{1}{2} \pm i \frac{\sqrt{1 - 4 T_{ij}}}{2}$ in correspondence to which the transient dynamics shows damped oscillations. The damping factor is determined by the real part $ Re(\lambda_{1,2})=- \frac{1}{2}$  and the oscillation frequencies are determined by the imaginary part $Im(\lambda_{1,2})= \frac{\sqrt{1 - 4T_{ij}}}{2}$.
\end{itemize}
\end{example}

The above theorem and  examples identify intervals for the parameters within which the behavior of the transient stability is unchanged. This provides robustness to our results and extend the analysis  to cases where the inertia, damping and synchronization parameters are uncertain.

%\begin{figure} [htb]
%\centering
%\includegraphics[width=.5\columnwidth]{node}
%\caption{Time plot  of the state of each TCL, namely temperature $x(t)$ (top row) and mode $y(t)$ (bottom row).}
%\label{fig:Fig1}
%\end{figure}

%\begin{figure} [htb]
%\centering
%\includegraphics[width=.5\columnwidth]{spiral}
%\caption{Time plot  of the state of each TCL, namely temperature $x(t)$ (top row) and mode $y(t)$ (bottom row).}
%\label{fig:Fig1}
%\end{figure}

\section{Multiple interconnected micro-grids}\label{sec:mmc}

Let us now consider a network $G=(V,E)$ of interconnected smart-grids, where $V$ is the set of nodes, and $E$ is the set of arcs.  Figure \ref{fig:graphr1}  displays an example of interconnection topology.  Nodes represent smart-grids units and arcs represent power lines interconnections. We use shades of gray to emphasize different levels of connectivity of the smart-grids. The connectivity of a grid is indicated by the degree of the node.  We recall that for undirected graphs the degree of a node is number of links with an extreme in node~$i$. We denote by $d_i$ the degree of node $i$.

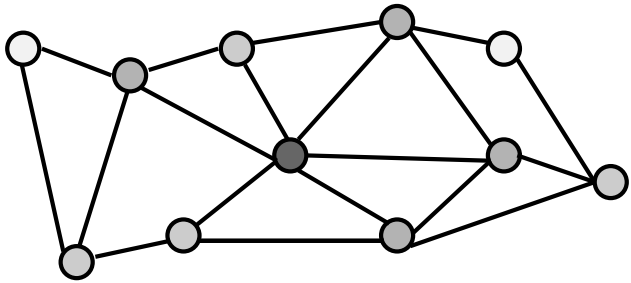
\begin{figure} [h]
\centering
\begin{tikzpicture}%[overlay]
\begin{scope}[shift={(0,0)},scale=.71]
\draw[fill=gray!60, ultra thick] (.15,0.8) circle (.3cm);
 \draw [ultra thick]  (-.1,.7) -- (-3.56,0.80); 
 \draw [ultra thick]  (-.1,1.01) -- (-1.6,3.1);
\draw [ultra thick]  (-.1,.7) -- (-1.6,-0.7);
%\draw (-.05,0.8) .. controls (.1,1.0) and (.2,.5) .. (.3,0.8);
%\node at (0,1.5){Yobe}; 
%
\begin{scope}[shift={(2,-.5)},scale=1,rotate=0]
\draw[fill=gray!40, ultra thick] (.15,0.8) circle (.3cm); \draw [ultra thick]  (-.15,.8) -- (-1.6,1.3);\draw [ultra thick]  (-.15,.8) -- (-1.6,3.1);\draw [ultra thick]  (-.15,.8) -- (-3.6,-0.4);
%\node at (0,1.8){Borno};
\end{scope}
\begin{scope}[shift={(-2,2.5)},scale=1,rotate=0]
\draw[fill=gray!60, ultra thick] (.15,0.8) circle (.3cm); \draw [ultra thick]  (-.15,.8) -- (-2.6,0.4);
%\node at (-.7,0.1){Adamawa};%\draw [thick]  (0.5,.0) -- (0.5,-0.2) -- (1.4,-0.2) -- (1.4,0.0);
\end{scope}
\begin{scope}[shift={(0,2)},scale=1,rotate=0]
\draw[fill=gray!10, ultra thick] (.15,0.8) circle (.3cm); \draw [ultra thick]  (-.1,0.9) -- (-1.6,1.2);
%\node at (-0.6,0.5){Taraba};
\end{scope}
\begin{scope}[shift={(-2,-1.5)},scale=1,rotate=0]
\draw[fill=gray!60, ultra thick] (.15,0.8) circle (.3cm); \draw [ultra thick]  (-.1,.7) -- (-3.6,0.7);\draw [ultra thick]  (-.1,.7) -- (-3.6,0.7);\draw [ultra thick]  (-.0,1.02) -- (-1.7,2.02);
%\node at (0,1.5){Gombe};
\end{scope}
\begin{scope}[shift={(-4,0)},scale=1,rotate=0]
\draw[fill=gray!120, ultra thick] (.15,0.8) circle (.3cm); \draw [ultra thick]  (-.1,.7) -- (-1.6,-0.5);\draw [ultra thick]  (-.1,.7) -- (-2.7,2.1);
\draw [ultra thick]  (.1,1.1) -- (-.7,2.5);\draw [ultra thick]  (.3,1.1) -- (2.0,3.0);
%\node at (0,1.5){Bauchi};
\end{scope}
\begin{scope}[shift={(-6,-1.5)},scale=1,rotate=0]
\draw[fill=gray!40, ultra thick] (.15,0.8) circle (.3cm); \draw [ultra thick]  (-.1,.7) -- (-1.5,0.4);
%\node at (0,1.5){Plateau};
\end{scope}
\begin{scope}[shift={(-8,-2.0)},scale=1,rotate=0]
\draw[fill=gray!40, ultra thick] (.15,0.8) circle (.3cm); \draw [ultra thick]  (.2,1.1) -- (1.1,3.99);
\draw [ultra thick]  (-.1,1.0) -- (-0.9,4.6);
%\node at (0,1.5){Kaduna};
\end{scope}
\begin{scope}[shift={(-7,1.5)},scale=1,rotate=0]
\draw[fill=gray!60, ultra thick] (.15,0.8) circle (.3cm); \draw [ultra thick]  (.5,0.9) -- (1.8,1.3);\draw [ultra thick]  (-.2,0.8) -- (-1.5,1.3);
%\node at (0,1.5){Kano};
\end{scope}
\begin{scope}[shift={(-5,2.0)},scale=1,rotate=0]
\draw[fill=gray!40, ultra thick] (.15,0.8) circle (.3cm); 
%\node at (0,1.5){Jigawa};
\end{scope}
\begin{scope}[shift={(-9,2.0)},scale=1,rotate=0]
\draw[fill=gray!10, ultra thick] (.15,0.8) circle (.3cm); 
%\node at (0,1.5){Katsina};
\end{scope}
%%%%%%
\end{scope}
\end{tikzpicture}
\caption{Graph topology indicating smart-grids and interconnections.}  \label{fig:graphr1}
\end{figure}

Building on  model (\ref{dynamicsm}) developed for the single grid, we derive the following macroscopic dynamics for the whole grid:
\begin{equation}\label{dynamicsadvsplit1}\nonumber 
\left[
\begin{array}{c}
\dot x^{(1)}_1\\
\vdots \\
\dot x^{(n)}_1\\
\dot x^{(1)}_2 \\
\vdots \\
\dot x^{(n)}_2\\
\end{array}
\right]
 = \left[
\begin{array}{cccccc}
-\frac{D_1}{M_1}& \hdots &0 &  \frac{1}{M_1} & \hdots & 0 \\
0 & \ddots &  0 & 0 & \ddots &  0\\
0 & \hdots &  -\frac{D_n}{M_n} & 0 & \hdots & \frac{1}{M_n}\\
- T_{11} & \hdots & T_{1n}  & 0 & \ddots & 0 \\
& \ddots & & & \ddots &\\
T_{n1} & \hdots & - T_{nn}  & 0 & \hdots & 0  
\end{array}
\right]\left[
\begin{array}{c}
 x^{(1)}_1\\
 \vdots \\
 x^{(n)}_1\\
 x^{(i)}_2\\
 \vdots
 \\
 x^{(n)}_2
\end{array}
\right]. 
\end{equation}

In the above set of equations, the block matrix 
\begin{equation}\label{dynamicsadvsplit11111}\nonumber
L:=\left[
\begin{array}{cccccccc}
T_{11} & \hdots & -T_{1n}   \\
& \ddots & \\
-T_{n1} & \hdots & T_{nn} 
\end{array}
\right]
\end{equation}
is the graph-Laplacian matrix. Given a weighted graph its components $L=[l_{ij}]_{i,j \in \{1,\ldots, n\}}$ are given by 
\begin{equation}\label{Lap}
l_{ij}=\left\{
\begin{array}{ll}
-T_{ij} & \mbox{if $i \not = j$},\\
\sum_{h=1,h \not = i} T_{ih} & \mbox{if $i=j$}.
\end{array}\right.
\end{equation}
Note that given a Laplacian matrix, its row-sums are zero, its diagonal entries are nonnegative, and its non-diagonal entries are nonpositive. 
The above set of equations can be rewritten in compact form as follows
\begin{equation}\label{dynamicsadvsplit}
\left[
\begin{array}{c}
\dot X_1\\
\dot X_2
\end{array}
\right]
 = 
 \underbrace{
 \left[
\begin{array}{cccccccccc}
-Diag \Big(\frac{D_i}{M_i} \Big) & Diag \Big(\frac{1}{M_i} \Big) \\
-L & 0  
\end{array}
\right]}_{\mathcal A}
\left[
\begin{array}{c}
 X_1\\
 X_2
\end{array}
\right]. 
\end{equation}
We also recall that $L=[l_{ij}]_{i,j \in \{1,\ldots, n\}}$
where for an unweighted and undirected graph we have 
\begin{equation}\label{Lap}
l_{ij}=\left\{
\begin{array}{ll}
-1 & \mbox{if $(i,j)$ is an edge and not self-loop},\\
d(i) & \mbox{if $i=j$},\\
0 & \mbox{otherwise.}
\end{array}\right.
\end{equation}

We are ready to establish the next result. Let us denote by $span\{\mathbf 1\}=\{ \xi \in \mathbb R^n:\, \exists \eta \in \mathbb R \, s.t.  \,  \xi=\eta \mathbf 1\}$. 
Furthermore, let the following consensus set be defined as
$$\mathcal C=\{\xi \in \mathbb R^n:\,\xi \in span\{\mathbf 1\},\, \min_j x_j(0) \leq \xi \leq \max_j x_j(0)\}.$$
  
\begin{theorem}
Let a network of homogeneous micro-grids be given, and set $D_i=D$ for all $i$. Let $M_i=1$ for all $i$, and $T_{ij}=1$ for any $(i,j) \in E$. Then dynamics (\ref{dynamicsadvsplit}) describes a consensus dynamics, i.e., 
$$\lim_{t \rightarrow \infty} X_i(t) = x_i^* \in \mathcal C, \quad i=1,2.$$ Furthermore, let $D> \sqrt{-4 \mu_i}$ then the consensus value vector $(x_1^*, x_2^*)^T$ is an asymptotically stable node. Vice versa, if  $D< \sqrt{-4 \mu_i}$ then $(x_1^*, x_2^*)^T$ is an asymptotically stable spiral.
\end{theorem}
\begin{proof}
 Let us start by finding the roots of $det(\lambda \mathbb I - \mathcal A)$. To this purpose, we recall that for any generic block matrix it holds
\begin{equation}\label{ABCD}
det( \left[
\begin{array}{cc}
A & B \\
C & D  
\end{array}
\right])=det(DA-BC), \quad \mbox{if $BD=DB$.}
\end{equation}
Then, from the above we have 
\begin{equation}\label{ABCD1}
\begin{array}{ll}
det (\lambda \mathbb I - \mathcal A ) = det \Big( \left[
\begin{array}{cc}
\lambda \mathbb I+Diag \Big(\frac{D_i}{M_i} \Big) & -Diag \Big(\frac{1}{M_i} \Big) \\
L & \lambda \mathbb I    
\end{array}
\right] \Big)
\\ \\
=det(\lambda^2 I +\lambda I \cdot Diag(\frac{D_i}{M_i}) +Diag(\frac{1}{M_i})\cdot L ).
\end{array}
\end{equation}
Under the homogeneity assumption $D_i=D$ for all $i$, we have
\begin{equation}\label{ABCD1}
\begin{array}{ll}
det \Big(\lambda^2 I +\lambda I \cdot Diag \Big(\frac{D_i}{M_i} \Big) +Diag \Big(\frac{1}{M_i} \Big)\cdot L \Big) \\ \\
= det\Big( (\lambda^2  +\lambda D) I  +  L \Big)  \\ \\
= \prod_{i=1}^n \Big((\lambda^2  +\lambda D) - \mu_i\Big),
\end{array}
\end{equation}
where $\mu_i$ is the $i$th eigenvalue of $-L$. The roots of (\ref{ABCD1}) can be obtained by solving $\lambda^2  +\lambda D - \mu_i = 0$ from which we have
\begin{equation}\label{ABCD22}
\begin{array}{ll}
\lambda_i^+ = \frac{-D + \sqrt{D^2 + 4 \mu_i}}{2}, \quad 
\lambda_i^- = \frac{-D - \sqrt{D^2 + 4 \mu_i}}{2}.
\end{array}
\end{equation}
From the above, after noting that the real part of the eigenvalues is negative,  we can conclude that system (\ref{dynamicsadvsplit}) is asymptotically stable. \end{proof}

\begin{remark}
The result stated in the above theorem applies also to the case where the micro-grids have different inertia but the same ratio $D=\frac{D_i}{M_i}$ for all $i \in V$.  In this case we need to consider the Laplacian matrix of the corresponding weighted graph $\tilde L = Diag(\frac{1}{M_i}) L$ and the associated eigenvalues. 
\end{remark}

We next recall some properties of the Laplacian spectrum and use such properties to investigate the insurgence of topology-induced oscillations.  The maximal eigenvalue $\tilde \mu_n$ of a symmetric Laplacian matrix $L=L^T$ in $\mathbb R^{n \times n}$ satisfies the following lower and upper bounds which are degree-dependent:
\begin{equation}\label{degree}
d_{\mbox{max}} \leq \tilde \mu_n \leq 2 d_{\mbox{max}},
\end{equation} 
where the maximum degree is $d_{\mbox{max}} = \max_{i \in {1,\ldots, n}} d_i $ \cite[Chapter 6]{lns-v.95}.  We also observe that the eigenvalues appearing in (\ref{ABCD22}) refer to the negative Laplacian, and therefore we have $\mu_i=-\tilde \mu_i$ for every eigenvalue  $\mu_i$ of the negative Laplacian $-L$ and  $\tilde \mu_i$ of the Laplacian $L$.

\begin{corollary}\label{cor1}
The following properties hold: 
\begin{itemize}
\item All eivenvalues $\lambda_i^+,\lambda_i^-$ for $i=1,\ldots, n$ are real and negative if $%\begin{equation}\label{degree}
D \geq \sqrt{8 d_{\mbox{max}} };$
 \item There exists at least one complex eigenvalue/eigenmode 
if  
 $%\begin{equation}\label{degree}
D \leq \sqrt{4 d_{\mbox{max}} }.
%\end{equation} 
$
\end{itemize}
\end{corollary}

\begin{corollary}
Given a chain topology of $n\geq 3$ nodes, for which $d_{\mbox{max}} = 2$  the following properties hold: 
\begin{itemize}
\item  All eivenvalues $\lambda_i^+,\lambda_i^-$ for $i=1,\ldots, n$ are real and negative if $%\begin{equation}\label{degree}
D \geq 4,$
 \item There exists at least one complex eigenvalue/eigenmode 
if 
 $%\begin{equation}\label{degree}
D \leq 2 \sqrt{2}.
%\end{equation} 
$
\end{itemize}
\end{corollary}

\subsection{Example of two interconnected micro-grids}

In this section, we specialize the above results to the case of two interconnected micro-grids. The interconnection topology is a chain one with two nodes, and the maximal degree is $d_{\mbox{max}}=1$. A graph representation is displayed in Fig. \ref{bs111}. 
\begin{figure} [h]
\centering
\def\svgwidth{.8\columnwidth}
\begin{tikzpicture}
\begin{scope}[shift={(0,0)},scale=.9]
\draw [->,very thick] (-1,0) -- (-.1,0);
\draw (0,0) circle (.1cm);
\draw [->,very thick] (.1,0) -- (1,0);
\draw [very thick] (1,-.6) -- (1,.6) -- (2,.6) -- (2,-.6) -- (1,-.6);
\draw [->,very thick] (2,0) -- (3.0,0);
\draw [very thick] (3.0,-.6) -- (3.0,.6) -- (5,.6) -- (5,-.6) -- (3.0,-.6);
\draw [->,very thick] (5,0) -- (6,0) -- (6,-2) --  (3.6,-2) -- (0,-2) -- (0,-.1);
\draw [->,very thick]  (6,0) -- (6.5,0) -- (6.5,-2.5) -- (-7.5,-2.5) -- (-7.5,-0)-- (-7.1,-0);
\node at (-0.2,0.3) {$+$};
\node at (-0.2,-0.3) {$-$};\node at (0.5,0.5) {$e_{ji}$};
\node at (1.5,0.0) {$\frac{T_{ij}}{s}$};
\node at (4.0,0.0) { $\frac{1}{M_j s + D_j}$};
\node at (5.5,0.5) {$f_j$};
%\node at (5.5,0.3) {$f_i$};
%\node at (2.7,-2) {$\psi(f_i,t)$};
\node at (2.4,.5) {$P_j$};
%\node at (2.99,.4) {$\omega$};
%\draw (2.8,0.0) circle (.1cm);
%\draw [->,very thick] (2.9,0) -- (3.4,0);
%\draw [->,very thick] (2.8,.7) -- (2.8,0.1);
%\draw [->,very thick] (2.4,0) -- (2.4,-1.4);
\begin{scope}[shift={(-7,0)},scale=1]
%\draw [->,very thick] (-1,0) -- (-.1,0);
\draw (0,0) circle (.1cm);
\draw [->,very thick] (.1,0) -- (1,0);
\draw [very thick] (1,-.6) -- (1,.6) -- (2,.6) -- (2,-.6) -- (1,-.6);
\draw [->,very thick] (2,0) -- (3.0,0);
\draw [very thick] (3.0,-.6) -- (3.0,.6) -- (5,.6) -- (5,-.6) -- (3.0,-.6);
\draw [->,very thick] (5,0) -- (6,0) -- (6,-2) --  (3.6,-2) -- (0,-2) -- (0,-.1);
%\draw [very thick] (2.0,-.6-2) -- (2.0,.6-2) -- (3.6,.6-2) -- (3.6,-.6-2) -- (2.0,-.6-2);
%\draw [->,very thick] (3.6,-2) -- (0,-2) -- (0,-.1) ;
\node at (-0.9,0.3) {$f_j$};\node at (-0.2,0.3) {$+$};
\node at (-0.2,-0.3) {$-$};\node at (0.5,0.5) {$e_{ij}$};
\node at (1.5,0.0) {$\frac{T_{ij}}{s}$};
\node at (4.0,0.0) {$\frac{1}{M_i s + D_i}$};
\node at (5.5,0.5) {$f_i$};
%\node at (5.5,0.3) {$f_i$};
%\node at (2.7,-2) {$\psi(f_i,t)$};
\node at (2.4,.5) {$P_i$};
\end{scope}
\end{scope}
      \end{tikzpicture}\caption{Block representation of two interconnected micro-grids.}\label{bs111}
\end{figure}
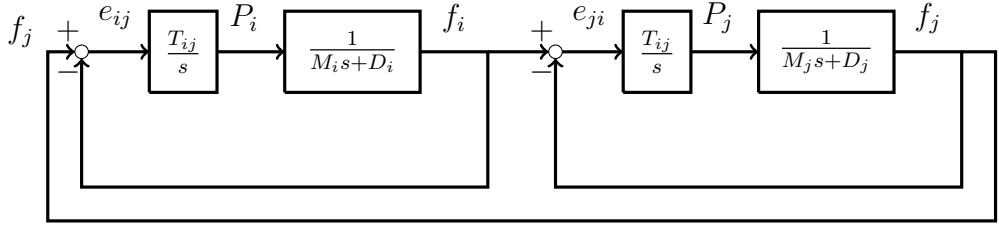

Dynamics (\ref{dynamicsadvsplit}) can be rewritten as 

 \begin{equation}\label{dynamicsadvsplit2nodes}
\left[
\begin{array}{c}
\dot x^{(i)}_1\\
\dot x^{(j)}_1\\
\dot x^{(i)}_2 \\
\dot x^{(j)}_2\\
\end{array}
\right]
 = \left[
\begin{array}{cccc}
-\frac{D_1}{M_1} &0 &  \frac{1}{M_1} & 0 \\
0 & -\frac{D_2}{M_2} & 0 & \frac{1}{M_2}\\
-T_{11} & T_{12}  & 0 & 0 \\
T_{12} & -T_{11}  & 0 & 0  
\end{array}
\right]\left[
\begin{array}{c}
 x^{(i)}_1\\
 x^{(j)}_1\\
 x^{(i)}_2\\
 x^{(j)}_2
\end{array}
\right]. 
\end{equation}
The Laplacian of the weighted graph is given by 
\begin{equation}\label{dynamicsadvsplit11111}
\tilde L = Diag(\frac{1}{M_i}) L=\left[
\begin{array}{cccccccc}
\frac{1}{M_1} & 0   \\
0 & \frac{1}{M_2}
\end{array}
\right] \left[
\begin{array}{cccccccc}
1& -1   \\
-1 & 1
\end{array}
\right].
\end{equation}

Assuming $D=\frac{D_1}{M_1}=\frac{D_2}{M_2}$, from Corollary~\ref{cor1} we infer that 
\begin{itemize}
\item All eivenvalues $\lambda_i^+,\lambda_i^-$ for $i=1,\ldots, n$ are real and negative if $D \geq \sqrt{8};$
 \item There exists at least one complex eigenvalue/eigenmode 
if   $ D \leq 2.$
\end{itemize}
In other words,  if the ratio between the damping coefficient and the inertia of each micro-grid is greater than $\sqrt{8}$ then we certainly have an overdamped dynamics, and observe no overshoots and no oscillations.   Differently, if the ratio between the damping coefficient and the inertia of each micro-grid is less than 2 then we certainly have an underdamped dynamics, and observe overshoots and oscillations. 
\section{Absolute stability}\label{sec:as}
In this section we extend the analysis to the case where both frequency and power flow measurements are subject to disturbances. Using a traditional technique in nonlinear analysis and control we isolate the nonlinearities in the feedback loop, and analyze stability under some mild assumptions on the nonlinear parameters. 

%\subsection{Microscopic model - smart grid $i$}
Likewise in the previous section we consider  two interconnected  micro-grids, and assume that each micro-grid can be described in terms of power flow $P_i$ and frequency $f_i$. Assuming disturbed measurements on $f_i$, the evolution of the power flow is given by 
\begin{equation}\label{dynamics}
%\left\{
\begin{array}{lll}
\dot P_i = T_{ij} (f_j-\psi(f_i))=T_{ij} \tilde e_{ij},%\\ \\
%\dot f_i = - \frac{D_i}{M_i} f_i + \frac{P_i}{M_i}
\end{array}%\right.
\end{equation}
where $T_{ij}$ is the synchronizing coefficient as in the previous section and where the new term $\psi(.)$ is a sector nonlinearity satisfying the following assumption.
\begin{assumption}\label{asm1} Function $\psi(f_i,t)$ is time-varying and satisfies the sector condition
$$0\leq \psi(f_i) \leq \tilde k f_i.$$
\end{assumption}
 %This coefficient is obtained as the inverse of the transmission reactance between micro-grid $i$ and $j$. 
%In other words,
 Now, the  power $P_i$ depends on a disturbed measure of the frequency error $\tilde e_{ij}:=f_j-\psi(f_i)$. %In particular, in response to a positive difference we have power  injected into micro-grid $i$ from micro-grid $j$. Vice versa, a negative difference induces power from micro-grid $i$ to $j$. 

The dynamics for $f_i$ still follows a traditional swing equation, which now involves disturbed measurements of the frequency $\psi(f_i) $ and of the power flow $\psi(P_i)$: 
\begin{equation}\label{dynamics1}
%\left\{
\begin{array}{lll}
%\dot P_i = T_{ij} (f_j-f_i)=T_{ij} e_{ij},\\ \\
\dot f_i = - \frac{D_i}{M_i} \psi(f_i) + \frac{\psi(P_i)}{M_i} + \omega.
\end{array}%\right.
\end{equation}
In the above model, $M_i$ and $D_i$ are the inertia and damping constants of the $i$th micro-grid, respectively.
Similarly to the previous section, we denote  
$f_i=x^{(i)}_1$, $P_i=x^{(i)}_2$, $f_j=x^{(j)}_1$, and by considering $f_j$ as an exogenous input to micro-grid $i$, the dynamics of micro-grid $i$ reduces to the following second-order system
\begin{equation}\label{dynamicsm1}
\begin{array}{lll}
\dot x= \left[
\begin{array}{c}
\dot x^{(i)}_1\\
\dot x^{(i)}_2
\end{array}
\right]
 = \underbrace{\left[
\begin{array}{cc}
-\frac{D_i}{M_i} & \frac{1}{M_i}\\
-T_{ij} & 0
\end{array}
\right]}_{A}
\left[
\begin{array}{c}
 \psi(x^{(i)}_1)\\
 \psi(x^{(i)}_2)
\end{array}
\right] 
\\ \qquad \qquad \qquad \qquad \qquad \quad  + 
\underbrace{\left[
\begin{array}{cc}
1 & 0\\
0 & T_{ij}
\end{array}
\right]}_{B} %x_i^{(j)}
\left[
\begin{array}{c}
\omega\\
x_i^{(j)}
\end{array}
\right].
\end{array}
\end{equation}

The block system of the $i$th micro-grid, which admits the state space representation (\ref{dynamicsm1}), is displayed in Fig.~\ref{bs}.
\begin{figure} [t]
\centering
\def\svgwidth{.8\columnwidth}
\begin{tikzpicture}
\draw [->,very thick] (-1,0) -- (-.1,0);
\draw (0,0) circle (.1cm);
\draw [->,very thick] (.1,0) -- (1,0);
\draw [very thick] (1,-.6) -- (1,.6) -- (2,.6) -- (2,-.6) -- (1,-.6);
\draw [->,very thick] (2,0) -- (2.7,0);
\draw [very thick] (3.4,-.6) -- (3.4,.6) -- (5,.6) -- (5,-.6) -- (3.4,-.6);
\draw [very thick] (5,0) -- (6,0) -- (6,-2) --  (3.6,-2);
\draw [very thick] (2.0,-.6-2) -- (2.0,.6-2) -- (3.6,.6-2) -- (3.6,-.6-2) -- (2.0,-.6-2);
\draw [->,very thick] (2.0,-2) -- (0,-2) -- (0,-.1) ;
\node at (-1.1,0.3) {$f_j$};\node at (-0.2,0.3) {$+$};
\node at (-0.2,-0.3) {$-$};\node at (0.5,0.3) {$e$};
\node at (1.5,0.0) {$\frac{T_{ij}}{s}$};
\node at (4.25,0.0) {$\frac{1}{M_i s + D_i}$};
\node at (5.5,0.3) {$f_i$};
%\node at (5.5,0.3) {$f_i$};
\node at (2.7,-2) {$\psi(f_i,t)$};
\node at (2.4,.3) {$P_i$};
\node at (2.99,.4) {$\omega$};
\draw (2.8,0.0) circle (.1cm);
\draw [->,very thick] (2.9,0) -- (3.4,0);
\draw [->,very thick] (2.8,.7) -- (2.8,0.1);
\draw [->,very thick] (2.4,0) -- (2.4,-1.4);
      \end{tikzpicture}\caption{Block system representing micro-grid $i$.}\label{bs}
\end{figure}
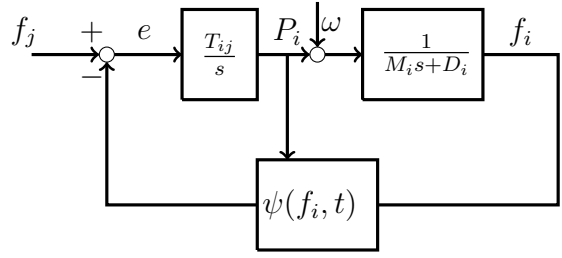

Building on the Kalman-Yakubovich-Popov lemma, absolute stability is linked to 
strictly positive realness of $Z(s) = \mathbb{I} + KG(s)$ where $K=\tilde k \mathbb I$ and $G(s)$ is the transfer function of linear part of system (\ref{dynamicsm1}) which is obtained as $G(s)=C^T[s\mathbb{I} - A]^{-1}B$ where we set
\begin{equation}\nonumber
\begin{array}{lll}
A=\left[\begin{array}{lll}
 -\frac{D_i}{M_i} & \frac{1}{M_i}\\
-T_{ij} & 0
\end{array} \right],
 & B=\left[\begin{array}{lll}
 1 & 0\\
0 & T_{ij}
\end{array} \right], & C=\left[\begin{array}{lll}
 1 & 0\\
0 & 1
\end{array} \right].
\end{array}
\end{equation}

 We recall from Theorem \ref{thm1} that matrix $A$ is Hurwitz.
%
%{\color{blue} Before addressing absolute stability, we first investigate conditions under which matrix $A$ is Hurwitz. To be Hurwitz, the trace of matrix $A$ must be negative, i.e. $Tr(A) = 2(r-3rx-\gamma-\alpha)<0$, and the determinant must be positive, i.e. $\Delta(A) = (r-3rx-\gamma-\alpha)^2 -  (-rx-\gamma)^2 > 0$. For the first condition, we can neglect the multiplier and have
%$$Tr(A) = r-3rx-\gamma-\alpha \le r(1 - 3/2) -\gamma -\alpha,$$
%where the equality holds for the fact that we are considering the fact that $x_1=x_2$ and thus at most it can be $0.5$. In the case where $x$ is sufficiently small, we can tune $\gamma$ and $\alpha$ to put a constraint on the inequality for which it holds $Tr(A) <0$. For the condition on the determinant, we can neglect the power and have
%$$\Delta(A) = 3rx+\gamma+\alpha-r-rx-\gamma  = 2rx + \alpha - r > 0,$$
%which holds true, for a similar reasoning as for the trace.}

The idea now is to isolate the nonlinearities in the feedback loop and introduce a new variable for them, say $\psi$. Let us first obtain the transfer function associated to the dynamical system (\ref{dynamicsm1}):
\begin{equation} \label{eq:cccscG}
\begin{array}{lll}
G(s) = C^T[s\mathbb{I} - A]^{-1}B \\ \\
= \frac{1}{s(s+\frac{D_i}{M_i})+\frac{T_{ij}}{M_i}}
%\begin{scriptsize}
\left[ \begin{array}{cc}
s & \frac{1}{M_i} \\
-T_{ij} & s+\frac{D_i}{M_i}  \end{array} \right] \cdot 
\left[ \begin{array}{cc}
1 & 0 \\
0 & T_{ij} \end{array} \right] \\ \\
= \frac{1}{s(s+\frac{D_i}{M_i})+\frac{T_{ij}}{M_i}}
%\begin{scriptsize}
\left[ \begin{array}{cc}
s & \frac{T_{ij}}{M_i} \\
-T_{ij} & (s+\frac{D_i}{M_i}) T_{ij} \end{array} \right]\\ \\
= \frac{1}{\Delta(s\mathbb I - A)}
%\begin{scriptsize}
\left[ \begin{array}{cc}
s & \frac{T_{ij}}{M_i} \\
-T_{ij} & (s+\frac{D_i}{M_i}) T_{ij} \end{array} \right], 
%\end{scriptsize}
\end{array}
\end{equation}
where $\Delta(s \mathbb I - A)=s(s+\frac{D_i}{M_i})+\frac{T_{ij}}{M_i}$. Then, for $Z(s)$ we obtain
\begin{equation} \label{eq:cccscZ}\nonumber
%\begin{scriptsize}
\begin{array}{lll}
Z(s) = \mathbb{I} + KG(s)  \\ \\ 
= \left[ \begin{array}{cc}
1 & 0 \\
0 & 1  \end{array} \right] +
 \frac{k}{s(s+\frac{D_i}{M_i})+\frac{T_{ij}}{M_i}}
%\begin{scriptsize}
\left[ \begin{array}{cc}
s & \frac{T_{ij}}{M_i} \\
-T_{ij} & (s+\frac{D_i}{M_i}) T_{ij} \end{array} \right]\\ \\
= \left[ \begin{array}{cc}
1 & 0 \\
0 & 1  \end{array} \right] + \frac{k}{\Delta(s\mathbb I - A)}
%\begin{scriptsize}
\left[ \begin{array}{cc}
s & \frac{T_{ij}}{M_i} \\
-T_{ij} & (s+\frac{D_i}{M_i}) T_{ij} \end{array} \right]
\\ \\
=
\frac{1}{\Delta(s\mathbb I - A)} \left[ \begin{array}{cc}
ks+\Delta(s\mathbb I - A) & k \frac{T_{ij}}{M_i} \\
- kT_{ij} & k (s+ \frac{D_i}{M_i}) T_{ij} +\Delta(s\mathbb I - A) \end{array} \right]
\\ \\
=
\frac{1}{s(s+\frac{D_i}{M_i})+\frac{T_{ij}}{M_i}} \\ \\
\cdot 
\begin{scriptsize}
\left[ \begin{array}{cc}
ks+s(s+\frac{D_i}{M_i})+\frac{T_{ij}}{M_i} & k \frac{T_{ij}}{M_i} \\ 
- kT_{ij} & k (s+ \frac{D_i}{M_i}) T_{ij} +s(s+\frac{D_i}{M_i})+\frac{T_{ij}}{M_i} \end{array} \right]
\end{scriptsize}
\\ \\ =
\frac{1}{s(s+\frac{D_i}{M_i})+\frac{T_{ij}}{M_i}} \\ \\
\cdot 
\begin{scriptsize}
\left[ \begin{array}{cc}
s^2 +(\frac{D_i}{M_i}+k)s +\frac{T_{ij}}{M_i} & k \frac{T_{ij}}{M_i}\\
- k T_{ij} & s^2 +(\frac{D_i}{M_i}+kT_{ij})s +\frac{T_{ij}}{M_i} + \frac{k T_{ij} D_i}{M_i}\end{array} \right].\end{scriptsize}
\end{array}
%\end{scriptsize}
\end{equation}
Note that matrix $A$ is Hurwitz. This implies that also $Z(s)$ is Hurwitz as the poles of $Z(s)$ coincide with the eigenvalues of $A$.  We use this in the proof of  absolute stability of the dynamical system (\ref{dynamicsm1}) established next.
 
\begin{theorem}\label{th5}
Let the dynamical  system (\ref{dynamicsm1}) be given where $A$ is Hurwitz. Furthermore, let us consider the sector nonlinearities as in Assumption \ref{asm1}. Then, $Z(s)$ is strictly positive real and system (\ref{dynamicsm1}) is absolutely stable. 
%Let us consider a Lyapunov function candidate $V(x) = x^TPx$. We can isolate the sector nonlinearity $\psi(t, y)$ in a feedback system, when $\dot{V}(t,x)$ is negative definite. By using the Kalman-Yakubovich-Popov lemma, this is equivalent to proving the strict positive realness of $Z(s) = \mathbb{I} + KC(s\mathbb{I} -A)^{-1}B$. Since $Z(s)$ is strictly positive real, then $\dot{V}(t,x)$ is negative definite.
\end{theorem}
\begin{proof}
We first prove that $Z(s)$ is strictly positive real.  For this to be true,  the following conditions must hold true:
\begin{itemize}
\item $Z(s)$ is Hurwitz, namely the poles of all entries of the matrix $Z(s)$ have negative real parts;
\item $ Z(j\omega) + Z(-j\omega) > 0, \quad \forall \omega \in \mathbb{R};$
\item $Z(\infty) + Z^T(\infty) >0$.
\end{itemize}
For the first condition note that $Z(s)$ is Hurwitz as its poles are the roots of $s(s+\frac{D_i}{M_i})+\frac{T_{ij}}{M_i}=0$, which coincide with the values  obtained in (\ref{eigss}) and which we rewrite here for convenience: $\lambda_{1,2} = \frac{1}{2} \Big(-\frac{D_i}{M_i} \pm \sqrt{(\frac{D_i}{M_i})^2 - 4 \frac{T_{ij}}{M_i}}\Big)$. As for the second condition, $ Z(j\omega) + Z(-j\omega) > 0, \quad \forall \omega \in \mathbb{R}$, let us obtain for $Z(j\omega)$ and $Z(-j\omega)$  the following expressions:
\begin{equation}\nonumber
\begin{array}{lll}
Z(j\omega)=
\frac{1}{
\frac{T_{ij}}{M_i}-\omega^2 +\frac{D_i}{M_i} j \omega} \\ \\
\cdot  
\begin{scriptsize}
\left[ \begin{array}{cc}
\frac{T_{ij}}{M_i}-\omega^2 +(\frac{D_i}{M_i}+k) j\omega & k \frac{T_{ij}}{M_i}\\
- k T_{ij} & \frac{T_{ij}}{M_i} + \frac{k T_{ij} D_i}{M_i}-\omega^2 +(\frac{D_i}{M_i}+kT_{ij}) j\omega \end{array} \right].
\end{scriptsize}
\end{array}
\end{equation}
\begin{equation}\nonumber
\begin{array}{lll}
Z(-j\omega)=
\frac{1}{
\frac{T_{ij}}{M_i}-\omega^2 -\frac{D_i}{M_i} j \omega}  \\ \\
\cdot  
\begin{scriptsize}
\left[ \begin{array}{cc}
\frac{T_{ij}}{M_i}-\omega^2 -(\frac{D_i}{M_i}+k) j\omega  & k \frac{T_{ij}}{M_i}\\
- k T_{ij} & \frac{T_{ij}}{M_i} + \frac{k T_{ij} D_i}{M_i}-\omega^2 -(\frac{D_i}{M_i}+kT_{ij}) j\omega \end{array} \right],
\end{scriptsize}
\end{array}
\end{equation}
By combining the expressions above for $Z(j\omega)$ and $Z(-j\omega)$ we then obtain
\begin{equation}
\begin{array}{lll}
Z(j\omega)+ Z(-j\omega)=
\frac{1}{
\Delta(j \omega \mathbb I -A)\Delta(-j \omega \mathbb I -A) }  
\Delta(-j \omega \mathbb I -A) \\ \\
\cdot 
\begin{scriptsize}
\left[ \begin{array}{cc}
\frac{T_{ij}}{M_i}-\omega^2 +(\frac{D_i}{M_i}+k) j\omega  & k \frac{T_{ij}}{M_i}\\
- k T_{ij} & \frac{T_{ij}}{M_i} + \frac{k T_{ij} D_i}{M_i}-\omega^2 +(\frac{D_i}{M_i}+kT_{ij}) j\omega \end{array} \right]
\end{scriptsize}
\\  \\
+
\Delta(j \omega \mathbb I -A) \\ \\
\cdot 
\begin{scriptsize}
\left[ \begin{array}{cc}
\frac{T_{ij}}{M_i}-\omega^2 -(\frac{D_i}{M_i}+k) j\omega & k \frac{T_{ij}}{M_i}\\
- k T_{ij} & \frac{T_{ij}}{M_i} + \frac{k T_{ij} D_i}{M_i}-\omega^2 -(\frac{D_i}{M_i}+kT_{ij}) j\omega \end{array} \right]
\end{scriptsize}
\\ \\
=
\frac{1}{
\Big(\frac{T_{ij}}{M_i}-\omega^2\Big)^2- \Big(\frac{D_i}{M_i}j\omega\Big)^2}
\left[ \begin{array}{cc}
z_{11} & z_{12} \\
z_{21} & z_{22}  \end{array} \right],
\end{array}
\end{equation}
where we set $z_{11}$ and $z_{22}$ as follows:
\begin{equation}\nonumber 
\begin{array}{lll}z_{11} &= 2[\omega^4 - 2 \omega^2 \frac{T_{ij}}{M_i} + \omega^2 \frac{D_i}{M_i} (\frac{D_i}{M_i} +k) + (\frac{T_{ij}}{M_i})^2]\\
&= 2[(\omega^2 - \frac{T_{ij}}{M_i} )^2 
+ \omega^2 \frac{D_i}{M_i} (\frac{D_i}{M_i} +k) ],\\\\
z_{22} &= 2[\omega^4 -  \omega^2 (\frac{T_{ij}}{M_i} +\frac{k T_{ij} D_i}{M_i})
+ \omega^2 \frac{D_i}{M_i} (\frac{D_i}{M_i} +k T_{ij}) \\
&
- \omega^2\frac{T_{ij}}{M_i}
+ \frac{T_{ij}}{M_i}(\frac{T_{ij}}{M_i}+\frac{kT_{ij}D_i}{M_i})+ \omega^2 \frac{D_i}{M_i} (\frac{D_i}{M_i} +k T_{ij})\\
%& = 2[\omega^4-\omega^2\frac{T_{ij}}{M_i} - \omega^2\frac{kT_{ij}D_i}{M_i}
%+ \omega^2\frac{kT_{ij}D_i}{M_i} \\
%&
%+ \omega^2 (\frac{D_i}{M_i})^2 - \omega^2\frac{T_{ij}}{M_i} - \frac{1}{2}\omega^2\frac{kT_{ij}D_i}{M_i}+(\frac{T_{ij}}{M_i})^2 + \frac{T_{ij}}{M_i}\frac{kT_{ij}D_i}{M_i}
%]\\
& = 2[(\omega^2-\frac{T_{ij}}{M_i})^2  - 
 \omega^2 \frac{k T_{ij} D_i}{M_i} \\
 & + \frac{T_{ij}}{M_i} \frac{k T_{ij}D_i}{M_i}
 + \omega^2 \frac{D_i}{M_i} (\frac{D_i}{M_i} +k T_{ij})]\\
&= 2[(\omega^2-\frac{T_{ij}}{M_i})^2 + \omega^2 
(\frac{D_i}{M_i})^2  + \frac{T_{ij}}{M_i} \frac{k T_{ij}D_i}{M_i}].
\end{array}
\end{equation}
From the above equation we then have 
\begin{equation}\nonumber
\begin{array}{lll}
Z(j\omega)+ Z(-j\omega)
=
\frac{1}{
\Big(\frac{T_{ij}}{M_i}-\omega^2\Big)^2- \Big(\frac{D_i}{M_i}j\omega\Big)^2}\\ \\
\cdot 
%\begin{scriptsize}
%\left[ \begin{array}{cc}
%2[(\omega^2 - \frac{T_{ij}}{M_i} )^2 
%+ \omega^2 \frac{D_i}{M_i} (\frac{D_i}{M_i} +k) ] & z_{12} \\
%z_{21} & 2[(\omega^2-\frac{T_{ij}}{M_i})^2 + \omega^2 
%(\frac{D_i}{M_i})^2  + \frac{T_{ij}}{M_i} \frac{k T_{ij}D_i}{M_i}]\end{array} \right]
%\end{scriptsize}
%\\ \\ 
\left[
\begin{array}{cc}
2[ (\omega^2 - \frac{T_{ij}}{M_i} )^2 
+ \omega^2 \frac{D_i}{M_i} (\frac{D_i}{M_i} +k) ] \\
z_{21} \end{array} \right. \\
\qquad \qquad \qquad \quad  \left. \begin{array}{cc} 
 z_{12} \\ 
2[(\omega^2-\frac{T_{ij}}{M_i})^2 + \omega^2 
(\frac{D_i}{M_i})^2  + \frac{T_{ij}}{M_i} \frac{k T_{ij}D_i}{M_i}]
\end{array} \right] 
\\ \\
\qquad >0, \mbox{for all $\omega$}.
\end{array}
\end{equation}

The last inequality follows from the trace of the above matrix being positive. To see this note that 
\begin{equation}
\begin{array}{lll}
 z_{11} + z_{22} \\ \\ 
 = 2[ (\omega^2 - \frac{T_{ij}}{M_i} )^2  + \omega^2 \frac{D_i}{M_i} (\frac{D_i}{M_i} +k) ]\\ \\ 
 + 2[(\omega^2-\frac{T_{ij}}{M_i})^2 + \omega^2 
(\frac{D_i}{M_i})^2  + \frac{T_{ij}}{M_i} \frac{k T_{ij}D_i}{M_i}] >0.
\end{array}
\end{equation}

As for the third condition, namely $Z(\infty) + Z^T(\infty) >0$, 
we have that 
\begin{equation}
\begin{array}{lll}
\lim_{\omega \rightarrow \infty} z_{12}  = \lim_{\omega \rightarrow \infty} z_{21} = 0,\\
 \lim_{\omega \rightarrow \infty} z_{11}  = \lim_{\omega \rightarrow \infty} z_{22} = 2.
\end{array}
\end{equation} 
Then we obtain that $Z(\infty) + Z^T(\infty) = 2 \mathbb I > 0$. We can conclude that also  the third condition is verified.

Now we wish to show that there exists a Lyapunov function $V(x) = x^T \Phi x$, where $\Phi = [\Phi_{ij}] \in \mathbb R^{2 \times 2}$ is symmetric. After differentiation with respect to time and using (\ref{dynamicsm1}) we obtain
%
%
%\begin{equation}\nonumber
%\begin{array}{lll}
%A=\left[\begin{array}{lll}
% -\frac{D_i}{M_i} & \frac{1}{M_i}\\
%-T_{ij} & 0
%\end{array} \right],
% & B=\left[\begin{array}{lll}
% 1 & 0\\
%0 & T_{ij}
%\end{array} \right], & C=\left[\begin{array}{lll}
% 1 & 0\\
%0 & 1
%\end{array} \right].
%\end{array}
%\end{equation}
\begin{equation} \label{eq:ccVdot} \nonumber
\begin{array}{lll}
\dot{V}(t,x)  = \dot{x}^T\Phi x + x^T \Phi x \\ \\ \quad
 = x^TA^T  \Phi x + x^T \Phi Ax - \psi^TB^T \Phi x - x^T \Phi B\psi\\ \\ \quad
= [x_1 \, x_2 ] 
\left[\begin{array}{lll}
 -\frac{D_i}{M_i} & \frac{1}{M_i}\\
-T_{ij} & 0
\end{array} \right]^T 
 \left[\begin{array}{lll}
 \Phi_{11} & \Phi_{12}\\
\Phi_{21} & \Phi_{22}
\end{array} \right] 
\left[\begin{array}{lll}
 x_1\\
x_2
\end{array} \right] \\ \\ \quad
+  [x_1 \, x_2 ] \left[\begin{array}{lll}
 \Phi_{11} & \Phi_{12}\\
\Phi_{21} & \Phi_{22}
\end{array} \right] \left[\begin{array}{lll}
 -\frac{D_i}{M_i} & \frac{1}{M_i}\\
-T_{ij} & 0
\end{array} \right] \left[\begin{array}{lll}
 x_1\\
x_2
\end{array} \right] \\ \\ \quad
- [\psi_1 \, \psi_2 ] \left[\begin{array}{lll}
 1 & 0\\
0 & T_{ij}
\end{array} \right] 
 \left[\begin{array}{lll}
 \Phi_{11} & \Phi_{12}\\
\Phi_{21} & \Phi_{22}
\end{array} \right]   \left[\begin{array}{lll}
 x_1\\
x_2
\end{array} \right] \\ \\ \quad
-  [x_1 \, x_2 ] \left[\begin{array}{lll}
 \Phi_{11} & \Phi_{12}\\
\Phi_{21} & \Phi_{22}
\end{array} \right]   \left[\begin{array}{lll}
 1 & 0\\
0 & T_{ij}
\end{array} \right] B 
\left[\begin{array}{lll}
\psi_1\\
\psi_2
\end{array} \right],
\end{array}
\end{equation}
where we denote $\psi(t,y) =  [\psi_1 \, \psi_2 ]^T$. From Assumption \ref{asm1} and the property of first and third sector nonlinearities we have $-2\psi^T(\psi- Ky) \ge 0$. Furthermore, from symmetry of matrices $P$ and $K=\tilde k \mathbb I$, the time derivative of the candidate Lyapunov function can be rewritten as 
\begin{equation} \label{eq:ccVdot2}\nonumber
\begin{array}{lll}
\dot{V}(t,x)  \le  x^T(A^T \Phi + PA)x - 2x^T \Phi B \psi -2\psi^T(\psi -Ky) \\ \\ \quad
 = x^T(A^T \Phi + \Phi A)x - 2x^T \Phi B\psi +2\psi^TKCx - 2\psi^T\psi \\ \\ \quad
 =  x^T(A^T \Phi + \Phi A)x + 2x^T(C^TK- \Phi B)\psi - 2\psi^T\psi \\ \\ \quad
 = [x_1 \, x_2 ] 
 \left( \left[\begin{array}{lll}
 -\frac{D_i}{M_i} & \frac{1}{M_i}\\
-T_{ij} & 0
\end{array} \right]^T 
 \left[\begin{array}{lll}
 \Phi_{11} & \Phi_{12}\\
\Phi_{21} & \Phi_{22}
\end{array} \right] \right.  \\ \\ \quad
\left.
+ \left[\begin{array}{lll}
 \Phi_{11} & \Phi_{12}\\
\Phi_{21} & \Phi_{22}
\end{array} \right]  \left[\begin{array}{lll}
 -\frac{D_i}{M_i} & \frac{1}{M_i}\\
-T_{ij} & 0
\end{array} \right]  \right)
\left[\begin{array}{lll}
 x_1\\
x_2
\end{array} \right] \\ \\ \quad 
+ 2 [x_1 \, x_2 ]  
\left( \left[\begin{array}{lll}
 k & 0\\
0 & k
\end{array} \right] - \left[\begin{array}{lll}
 \Phi_{11} & \Phi_{12}\\
\Phi_{21} & \Phi_{22}
\end{array} \right] 
\left[\begin{array}{lll}
 1 & 0\\
 0 & T_{ij} 
\end{array} \right] \right) 
\left[\begin{array}{lll}
 \psi_1\\
 \psi_2
\end{array} \right] \\ \\ \quad
- 2 [\psi_1 \, \psi_2] \left[\begin{array}{lll}
 \psi_1\\
 \psi_2
\end{array} \right].
\end{array}
\end{equation}
The right-hand side of the above inequality is negative if there exist matrices $\Pi \in \mathbb R^{2 \times 2}$ and a positive scalar $\epsilon$ such that 
\begin{equation} \label{eq:ccPL}
\left\{
\begin{array}{lll}
A^T \Phi + \Phi A  =   - \Pi^T \Pi -\epsilon \Phi, \\
 \Phi  B  =  C^TK -\sqrt{2} \Pi^T,
\end{array}\right.
\end{equation}
or in explicit form 
\begin{equation} \label{eq:ccPL}\nonumber
%\left\{
\begin{array}{lll}
\left[\begin{array}{lll}
 -\frac{D_i}{M_i} & \frac{1}{M_i}\\
-T_{ij} & 0
\end{array} \right]^T 
 \left[\begin{array}{lll}
 \Phi_{11} & \Phi_{12}\\
\Phi_{21} & \Phi_{22}
\end{array} \right]   \\ \\ \quad
+ \left[\begin{array}{lll}
 \Phi_{11} & \Phi_{12}\\
\Phi_{21} & \Phi_{22}
\end{array} \right]  \left[\begin{array}{lll}
 -\frac{D_i}{M_i} & \frac{1}{M_i}\\
-T_{ij} & 0
\end{array} \right]  \\ \\
\quad  =   - \left[\begin{array}{lll}
 \Pi_{11} & \Pi_{12}\\
\Pi_{21} & \Pi_{22}
\end{array} \right]^T \left[\begin{array}{lll}
 \Pi_{11} & \Pi_{12}\\
\Pi_{21} & \Pi_{22}
\end{array} \right] % \\ \\
 %\quad 
 -\epsilon  \left[\begin{array}{lll}
 \Phi_{11} & \Phi_{12}\\
\Phi_{21} & \Phi_{22}
\end{array} \right]  , \\ \\
\left[\begin{array}{lll}
 \Phi_{11} & \Phi_{12}\\
\Phi_{21} & \Phi_{22}
\end{array} \right]   
\left[\begin{array}{lll}
1 & 0\\
0 & T_{ij}
\end{array} \right]  =  \left[\begin{array}{lll}
k & 0\\
0 & k
\end{array} \right] -\sqrt{2}  \left[\begin{array}{lll}
 \Pi_{11} & \Pi_{21}\\
\Pi_{12} & \Pi_{22}
\end{array} \right],
\end{array}%\right.
\end{equation}

%
% To specialise the result in our case, since $B = C = I$, the above condition can be rewritten also as
%\begin{equation} \label{eq:ccPL}
%\begin{array}{lll}
%A^TP + &PA  =   - L^TL -\epsilon P, \\
%& P  =  K -\sqrt{2}L^T.
%\end{array}
%\end{equation}

By introducing the solutions of the above in terms of $\Phi$, $\Pi$ and $\epsilon$, the time derivative of the candidate Lyapunov function can be rewritten as 
\begin{equation} \label{eq:ccVdot3}\nonumber 
\begin{array}{lll}
\dot{V}(t,x) \le -\epsilon x^T \Phi x - x^T \Pi^T \Pi  x + 2\sqrt{2}x^T \Pi^T\psi - 2\psi^T\psi \\ \\
\quad  = -\epsilon x^T \Phi x - [\Pi x - \sqrt{2}\psi]^T[\Pi x - \sqrt{2}\psi] \\ \\ \quad
 \le -\epsilon x^T \Phi x \\ \\ \quad 
 -\epsilon [x_1 \, x_2] 
 \left[\begin{array}{lll}
 \Phi_{11} & \Phi_{21}\\
\Phi_{12} & \Phi_{22}
\end{array} \right] 
\left[\begin{array}{lll}
 x_1\\
 x_2
\end{array} \right].
\end{array}
\end{equation}
It is well known that from the Kalman-Yakubovich-Popov lemma, there exist solutions in terms of  $\Phi$, $\Pi$, and $\epsilon$ satisfying the above set of matrix equalities, as the transfer function $Z(s)$ is positive real and this concludes our proof.
%This will be the aim of the experiments in the following section, for \emph{Linear Matrix Equalities}. 
\end{proof}

\begin{remark}
The above theorem has been obtained under the hypothesis that  both frequency and power flow measurements are subject to disturbances. The same result extend straightforwardly also to the case where the model parameters $T_{ij}$,  $M_i$ and  $D_i$
%(synchronization coefficient $T_{ij}$,  inertial coefficient $M_i$ and damping coefficient $D_i$) 
 are uncertain.   
 \end{remark}

\section{Simulations}\label{sec:sim}

This section provides simulation studies to corroborate the theoretical results developed in the previous sections. The analysis is based on open source data relating to a part of the Nigerian grid obtained from \cite{AO13}. The data set shows the one-line diagram of part of the distribution network including the geographical location of generators and load buses.  Figure~\ref{fig:one-line} displays the one-line diagram with the geographical names.
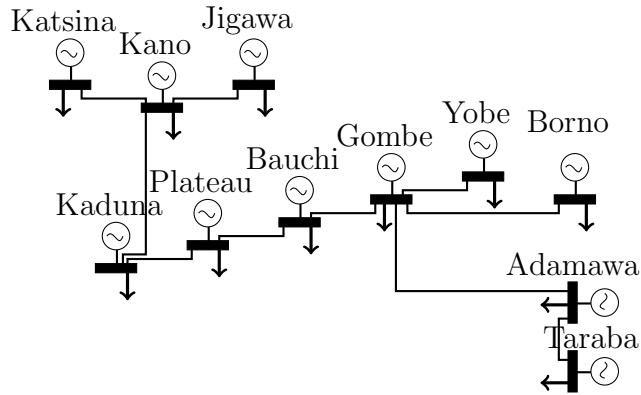
\begin{figure} [h]
\centering
\begin{tikzpicture}%[overlay]
\begin{scope}[shift={(0,0)},scale=.61]
\draw (.15,0.8) circle (.3cm); \draw [thick]  (.17,.5) -- (.17,0);
\draw (-.05,0.8) .. controls (.1,1.0) and (.2,.5) .. (.3,0.8);
\draw[fill=black] (-.3,0) rectangle (0.6,.2);\draw [very thick, ->]  (.4,0) -- (.4,-0.6);
\node at (0,1.5){Yobe}; \draw [thick]  (-0.2,.0) -- (-0.2,-0.2) -- (-1.6,-0.2) -- (-1.6,-0.3);
\begin{scope}[shift={(2,-.5)},scale=1,rotate=0]
\draw (.15,0.8) circle (.3cm); \draw [thick]  (.17,.5) -- (.17,0);
\draw (-.05,0.8) .. controls (.1,1.0) and (.2,.5) .. (.3,0.8);
\draw[fill=black] (-.3,0) rectangle (0.6,.2);\draw [very thick, ->]  (.4,0) -- (.4,-0.6);
\node at (0,1.8){Borno};\draw [thick]  (-0.2,.0) -- (-0.2,-0.2) -- (-3.5,-0.2) -- (-3.5,-0.0);
\end{scope}
\begin{scope}[shift={(2,-2.5)},scale=1,rotate=-90]
\draw (.15,0.8) circle (.3cm); \draw [thick]  (.17,.5) -- (.17,0);
\draw (-.05,0.8) .. controls (.1,1.0) and (.2,.5) .. (.3,0.8);
\draw[fill=black] (-.3,0) rectangle (0.6,.2);\draw [very thick, ->]  (.2,0) -- (.2,-0.6);
\node at (-.7,0.1){Adamawa};\draw [thick]  (0.5,.0) -- (0.5,-0.2) -- (1.4,-0.2) -- (1.4,0.0);
\end{scope}
\begin{scope}[shift={(2,-4)},scale=1,rotate=-90]
\draw (.15,0.8) circle (.3cm); \draw [thick]  (.17,.5) -- (.17,0);
\draw (-.05,0.8) .. controls (.1,1.0) and (.2,.5) .. (.3,0.8);
\draw[fill=black] (-.3,0) rectangle (0.6,.2);\draw [very thick, ->]  (.4,0) -- (.4,-0.6);
\node at (-0.6,0.5){Taraba};
\end{scope}
\begin{scope}[shift={(-2,-.5)},scale=1,rotate=0]
\draw (.15,0.8) circle (.3cm); \draw [thick]  (.17,.5) -- (.17,0);
\draw (-.05,0.8) .. controls (.1,1.0) and (.2,.5) .. (.3,0.8);
\draw[fill=black] (-.3,0) rectangle (0.6,.2);\draw [very thick, ->]  (.0,0) -- (.0,-0.6);
\node at (0,1.5){Gombe};\draw [thick]  (-0.2,.0) -- (-0.2,-0.2) -- (-1.6,-0.2) -- (-1.6,-0.3);
\draw [thick]  (0.25,.0) -- (0.25,-1.9) -- (4.0,-1.9);
\end{scope}
\begin{scope}[shift={(-4,-1)},scale=1,rotate=0]
\draw (.15,0.8) circle (.3cm); \draw [thick]  (.17,.5) -- (.17,0);
\draw (-.05,0.8) .. controls (.1,1.0) and (.2,.5) .. (.3,0.8);
\draw[fill=black] (-.3,0) rectangle (0.6,.2);\draw [very thick, ->]  (.4,0) -- (.4,-0.6);
\node at (0,1.5){Bauchi};\draw [thick]  (-0.2,.0) -- (-0.2,-0.2) -- (-1.6,-0.2) -- (-1.6,-0.3);
\end{scope}
\begin{scope}[shift={(-6,-1.5)},scale=1,rotate=0]
\draw (.15,0.8) circle (.3cm); \draw [thick]  (.17,.5) -- (.17,0);
\draw (-.05,0.8) .. controls (.1,1.0) and (.2,.5) .. (.3,0.8);
\draw[fill=black] (-.3,0) rectangle (0.6,.2);\draw [very thick, ->]  (.4,0) -- (.4,-0.6);
\node at (0,1.5){Plateau};\draw [thick]  (-0.2,.0) -- (-0.2,-0.2) -- (-1.6,-0.2) -- (-1.6,-0.3);
\end{scope}
\begin{scope}[shift={(-8,-2.0)},scale=1,rotate=0]
\draw (.15,0.8) circle (.3cm); \draw [thick]  (.17,.5) -- (.17,0);
\draw (-.05,0.8) .. controls (.1,1.0) and (.2,.5) .. (.3,0.8);
\draw[fill=black] (-.3,0) rectangle (0.6,.2);\draw [very thick, ->]  (.4,0) -- (.4,-0.6);
\node at (0,1.5){Kaduna};
\end{scope}
\begin{scope}[shift={(-7,1.5)},scale=1,rotate=0]
\draw (.15,0.8) circle (.3cm); \draw [thick]  (.17,.5) -- (.17,0);
\draw (-.05,0.8) .. controls (.1,1.0) and (.2,.5) .. (.3,0.8);
\draw[fill=black] (-.3,0) rectangle (0.6,.2);\draw [very thick, ->]  (.4,0) -- (.4,-0.6);
\node at (0,1.5){Kano};\draw [thick]  (-0.2,.0) -- (-0.2,-3.1) -- (-0.7,-3.1) -- (-0.7,-3.3);
\end{scope}
\begin{scope}[shift={(-5,2.0)},scale=1,rotate=0]
\draw (.15,0.8) circle (.3cm); \draw [thick]  (.17,.5) -- (.17,0);
\draw (-.05,0.8) .. controls (.1,1.0) and (.2,.5) .. (.3,0.8);
\draw[fill=black] (-.3,0) rectangle (0.6,.2);\draw [very thick, ->]  (.4,0) -- (.4,-0.6);
\node at (0,1.5){Jigawa};\draw [thick]  (-0.2,.0) -- (-0.2,-0.2) -- (-1.6,-0.2) -- (-1.6,-0.3);
\end{scope}
\begin{scope}[shift={(-9,2.0)},scale=1,rotate=0]
\draw (.15,0.8) circle (.3cm); \draw [thick]  (.17,.5) -- (.17,0);
\draw (-.05,0.8) .. controls (.1,1.0) and (.2,.5) .. (.3,0.8);
\draw[fill=black] (-.3,0) rectangle (0.6,.2);\draw [very thick, ->]  (.0,0) -- (.0,-0.6);
\node at (0,1.5){Katsina};\draw [thick]  (0.4,.0) -- (0.4,-0.2) -- (1.8,-0.2) -- (1.8,-0.3);
\end{scope}
%%%%%%
\end{scope}
\end{tikzpicture}
\caption{One-line diagram of part of Nigerian grid \cite{AO13}.}  \label{fig:one-line}
\end{figure}

From the one-line diagram we obtain the graph representation showing the interconnection between bus loads as in Fig.~\ref{fig:graphr}. The graph is characterized by $11$ nodes and $10$ arcs. Most nodes have degree $1$ or $2$ except for Gombe, and Kano  which have degree 4, and 3, respectively. The graph is undirected, i.e. the influence of smart-grid $i$ on $j$ is bidirectional. 
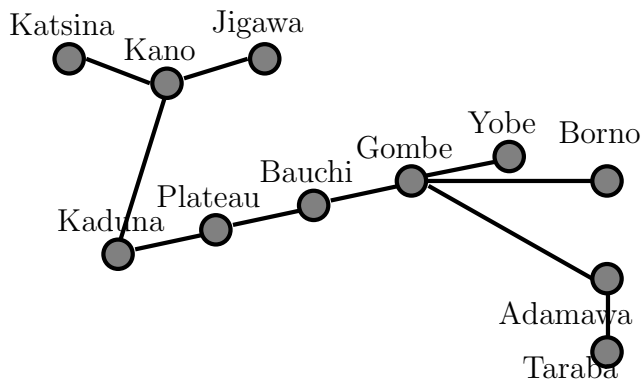
\begin{figure} [h]
\centering
\begin{tikzpicture}%[overlay]
\begin{scope}[shift={(0,0)},scale=.65]
\draw[fill=gray, ultra thick] (.15,0.8) circle (.3cm); \draw [ultra thick]  (-.1,.7) -- (-1.6,0.4);
%\draw (-.05,0.8) .. controls (.1,1.0) and (.2,.5) .. (.3,0.8);
\node at (0,1.5){Yobe}; 
\begin{scope}[shift={(2,-.5)},scale=1,rotate=0]
\draw[fill=gray, ultra thick] (.15,0.8) circle (.3cm); \draw [ultra thick]  (-.15,.8) -- (-3.6,0.8);\node at (0,1.8){Borno};
\end{scope}
\begin{scope}[shift={(2,-2.5)},scale=1,rotate=0]
\draw[fill=gray, ultra thick] (.15,0.8) circle (.3cm); \draw [ultra thick]  (-.15,.8) -- (-3.5,2.7);
\node at (-.7,0.1){Adamawa};%\draw [thick]  (0.5,.0) -- (0.5,-0.2) -- (1.4,-0.2) -- (1.4,0.0);
\end{scope}
\begin{scope}[shift={(2,-4)},scale=1,rotate=0]
\draw[fill=gray, ultra thick] (.15,0.8) circle (.3cm); \draw [ultra thick]  (.17,1.1) -- (.17,2.0);
\node at (-0.6,0.5){Taraba};
\end{scope}
\begin{scope}[shift={(-2,-.5)},scale=1,rotate=0]
\draw[fill=gray, ultra thick] (.15,0.8) circle (.3cm); \draw [ultra thick]  (-.1,.7) -- (-1.5,0.4);
\node at (0,1.5){Gombe};
\end{scope}
\begin{scope}[shift={(-4,-1)},scale=1,rotate=0]
\draw[fill=gray, ultra thick] (.15,0.8) circle (.3cm); \draw [ultra thick]  (-.1,.7) -- (-1.5,0.4);\node at (0,1.5){Bauchi};
\end{scope}
\begin{scope}[shift={(-6,-1.5)},scale=1,rotate=0]
\draw[fill=gray, ultra thick] (.15,0.8) circle (.3cm); \draw [ultra thick]  (-.1,.7) -- (-1.5,0.4);
\node at (0,1.5){Plateau};
\end{scope}
\begin{scope}[shift={(-8,-2.0)},scale=1,rotate=0]
\draw[fill=gray, ultra thick] (.15,0.8) circle (.3cm); \draw [ultra thick]  (.2,1.1) -- (1.1,3.99);
\node at (0,1.5){Kaduna};
\end{scope}
\begin{scope}[shift={(-7,1.5)},scale=1,rotate=0]
\draw[fill=gray, ultra thick] (.15,0.8) circle (.3cm); \draw [ultra thick]  (.5,0.9) -- (1.8,1.3);\draw [ultra thick]  (-.2,0.8) -- (-1.5,1.3);
\node at (0,1.5){Kano};
\end{scope}
\begin{scope}[shift={(-5,2.0)},scale=1,rotate=0]
\draw[fill=gray, ultra thick] (.15,0.8) circle (.3cm); 
\node at (0,1.5){Jigawa};
\end{scope}
\begin{scope}[shift={(-9,2.0)},scale=1,rotate=0]
\draw[fill=gray, ultra thick] (.15,0.8) circle (.3cm); 
\node at (0,1.5){Katsina};
\end{scope}
%%%%%%
\end{scope}
\end{tikzpicture}
\caption{Graph representation of part of Nigerian grid \cite{AO13}.}  \label{fig:graphr}
\end{figure}

 The numerical studies involve two sets of simulations. 
The first set of simulations  has been conducted %using the algorithm displayed in Table \ref{fig:algorithm} and 
considering  the following normalized parameters: number of smart-grids $n=11$, damping constant $D=1,3,6$ for three consecutive runs of simulations; Inertial constant $M=1$; Synchronizing coefficient $T=1$; Horizon window involving $N=500$ iterations; Step size $dt=.01$. The parameters and the dynamics is normalized in an interval $[0,1]$. For instance the initial state of each grid is a randomized bidimensional vector in the interval $[0,1]$. To simulate periodic disturbances, the initial state is reinitialized every 10 sec. To obtain realistic plots we rescale the state variable around $50$ Hz for the frequency and $30$ MWh for the power flow. 
 \begin{figure} [h]
\centering
\includegraphics[width=1.0\columnwidth]{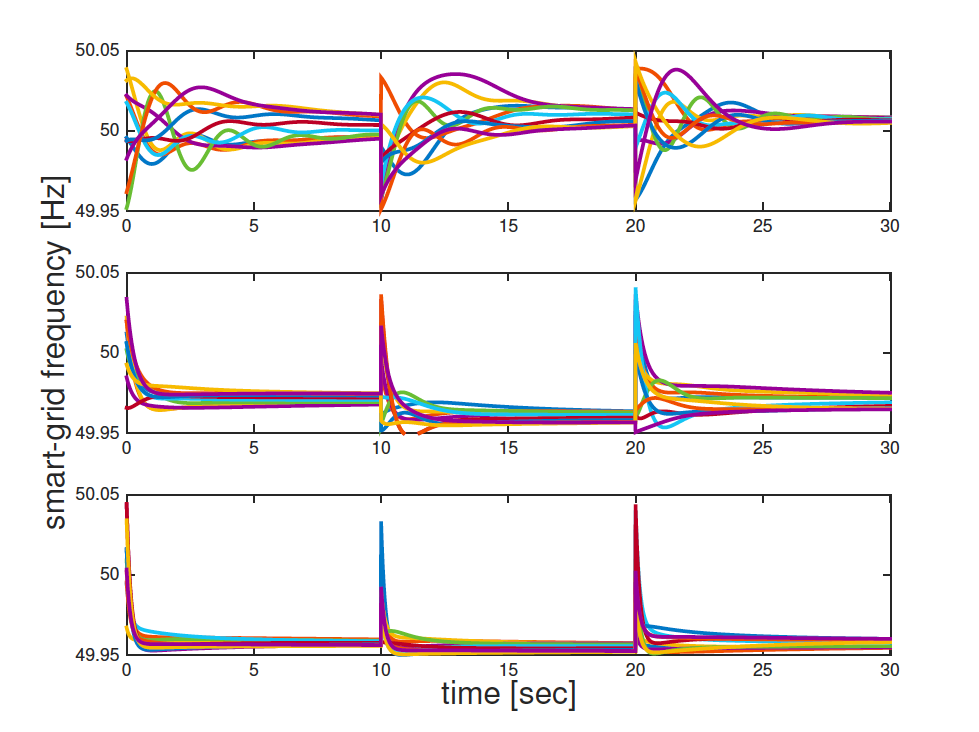}
\caption{Time series of smart-grids frequencies in Hz. }
\label{fig:Fig11}
\end{figure}
Figure \ref{fig:Fig11} displays the evolution of the frequency of each smart-grid. Frequencies are measured in Hz and are centered around $50$ Hz which is the nominal value. Oscillations remain within $1 \%$ of the nominal value, i.e., in  the interval $[49.95,50.05]$. From top to bottom we consider an increasing damping constant $D=1,3,6$ which reflects in damped oscillations and smaller time constants.
 \begin{figure} [h]
\centering
\includegraphics[width=1.0\columnwidth]{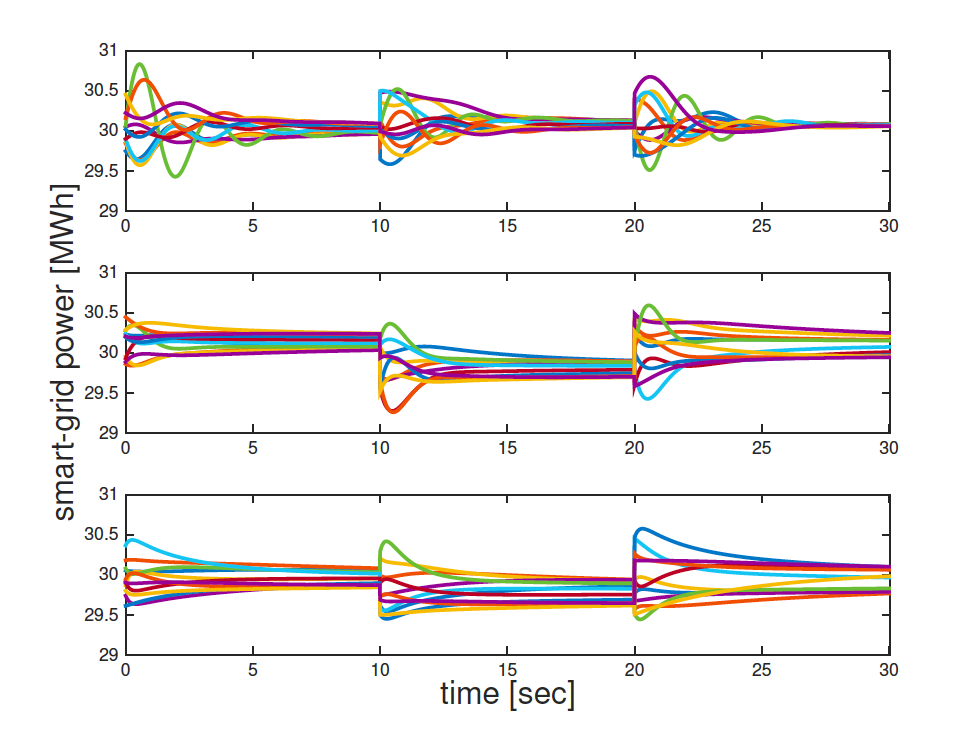}
\caption{Time series of smart-grids power flows in MWh. }
\label{fig:Fig12}
\end{figure}
Figure \ref{fig:Fig12} displays the evolution of the power flows in each smart-grid. Power flows are measured in MWh and are centered around the nominal value of $30$ MWh. From the plots we observe that oscillations remain within $3.3 \%$ of the nominal value, i.e., in  the interval $[29.00,31.00]$ MWh. From top to bottom the  damping constant is increasing and equal to $D=1,3,6$.
Note that the maximal degree of the network is $d_{\mbox{max}}=4$ and therefore for $D=1,3$ we have $D < \sqrt{4 d_{\mbox{max}} } = 4$ and  oscillations emerge as we have complex eigenvalues for $\lambda_i^+,\lambda_i^-,$ for $A$. Unlikewise for $D=6$ it holds  $D > \sqrt{8 d_{\mbox{max}} } = \sqrt{32} $ and therefore no oscillations and no complex eigenvalues emerge. The maximal eigenvalue of the Laplacian has been obtained as $\tilde \mu_n=5.1748$.

In a second set of simulations, we isolate one smart-grid from the rest of the power network and investigate the transient response under disturbances in the measurement of  frequency and power. Such disturbances are modeled using the paradigm developed in Section \ref{sec:as}. In particular we consider a first and third quadrant nonlinearity in the feedback loop. The function is periodic and we take for it the expression $\psi(t)= 1+ sin(\xi f t)$ in $[0,2]$, where $f$ is the frequency, $t$ is time, and $\xi$ is  a factor increasing the periodicity of the oscillation.  For the second set of simulations  we consider the following normalized parameters: number of smart-grids $n=1$, damping constant $D=1$; Inertial constant $M=1$; Synchronizing coefficient $T=1$; periodicity factor $\xi=1,5,10$ for three consecutive runs of simulations; Horizon window involving $N=1000$ iterations; Step size $dt=.01$; The initial state of each grid is a randomized bidimensional vector in the interval $[0,1]$. Both variables are rescaled  around $50$ Hz for the frequency and $30$ MWh for the power flow. 
\begin{figure} [h]
\centering
\includegraphics[width=1.0\columnwidth]{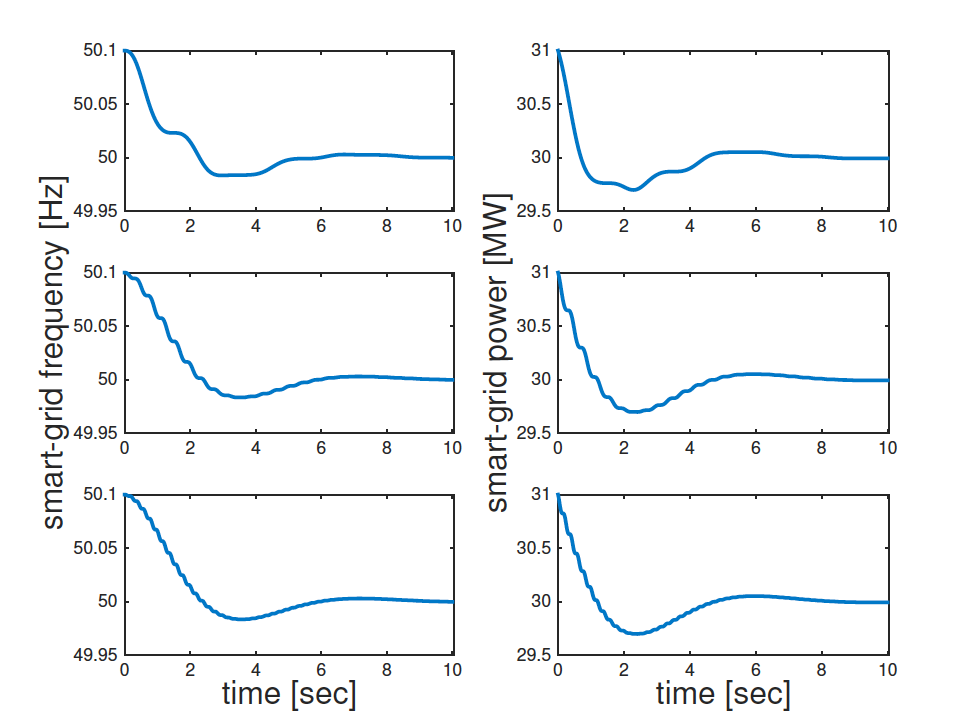}
\caption{Time series of smart-grids power flows in MWh. }
\label{fig:Fig2}
\end{figure}
Figure \ref{fig:Fig2} displays the time evolution of the frequency of each smart-grid (left) and power flow (right). As in the previous simulation example, frequencies are measured in Hz and are centered around $50$~Hz which is the nominal value. We observe that oscillations remain within $1 \%$ of the nominal value, i.e., in  the interval $[49.95,50.05]$. From top to bottom the damping constant is $D=1,3,5$ and this implies  a higher damping, smaller time constants, and faster convergence.
 %Figure \ref{fig:Fig} (right) displays the evolution of the power flows in each smart-grid. 
 Power flows are measured in MWh and are centered around the nominal value of $30$ MWh. The plots show that oscillations remain within $3.3 \%$ of the nominal value, i.e., in  the interval $[29.00,31.00]$ MWh. From top to bottom the  damping constant is increasing and equal to $D=1,3,5$.
%\begin{table}\normalsize
%\begin{center}
%%\begin{tabular}{p{12.5cm}}\\
%\begin{tabular}{p{8cm}}\\
%\toprule
%%\textbf{Algorithm} \\
%%\midrule
%\textbf{Input:} Set of parameters as in Table \ref{t:data1}.  \\
%\textbf{Output:} Expected profit $w_k^{(1)}$, error $w_k^{(2)}$, Shapley allocation $u_k^{(1)}$,  allocation $u_k^{(2)}$, and state $x_k$\\
%$\quad 1: $ \textbf{Initialize.} Compute $\tilde y_k$ as in Fig. \ref{fig:Fig14}\\%as $n$ random samples
%%from \\ $\quad \quad$ Gaussian distribution with mean $\bar m_0$\\
%%$\quad \quad$  and standard deviation $std(m_0)$,\\
%$\quad 2: $ \textbf{for} time $k=0,1,\ldots,T-1$ \textbf{do}\\
%$\quad 3: \quad$  generate $y_k$ from~(\ref{yk1})\\
%$\quad 4: \quad$ compute expected profit $w_k^{(1)}$ from (\ref{wk1})\\
%$\quad 5: \quad$ compute deviation $w_k^{(2)}$ from (\ref{cpp})\\
%$\quad 6: \quad$ compute Shapley allocation  $u_k^{(1)}$ from (\ref{shapley})\\
%$\quad 7: \quad$ compute  allocation  $u_k^{(2)}$  from LP (\ref{feask})\\
%$\quad 8: \quad$ compute  $x_{k+1}$ by running (\ref{rec2})\\
%%$\quad 8: \quad$ \textbf{end for}\\
%$\quad 9: $ \textbf{end for} \\
%$\quad 10: $ \textbf{STOP}\\
%\bottomrule
%\end{tabular}
%\end{center}\caption{Simulation algorithm} \label{fig:algorithm}
%\end{table}

\section{Discussion and conclusions} \label{sec:conc}
For single and multiple interconnected micro-grids, we have studied transient stability, namely the capability of the micro-grids to remain in synchronism even under  cyber-attacks or model uncertainties. First we have showed that transient dynamics can be robustly classified depending on specific intervals for the micro-grid parameters, such as synchronization, inertia, and damping parameters. We have then turned to study the analogies with consensus dynamics. We have  obtained bounds on the damping coefficient which  determine wether the network dynamics is underdamped or overdamped. Such a result is meaningful as in the case of underdamped dynamics we observe oscillation around the consensus value, whereas in the case of underdamped dynamics we observe a deviation of the consensus value from the nominal mains frequency.  The bounds are linked to the connectivity of the network.  We have also extended the stability analysis to the case of disturbed measurements due to hackering or parameter uncertainties. Using traditional nonlinear analysis and the Kalman-Yakubovich-Popov lemma we have first isolated the nonlinear terms in the feedback loop and have showed that nonlinearities do not compromise the stability of the system. %The theoretical findings have been illustrated on a case study on the Nigerian grid. 

There are three key directions for future work. First we wish to relax constraints on the nature of the disturbances. Indeed here we have assumed that such disturbances can be modeled using first and third quadrant nonlinearities.  Such nonlinearities tends to vanish around the equilibrium points. In reality, disturbances due to hackering can impact the systems even at the equilibrium, thus leading to synchronization deficiency.  A second direction involves the analysis of the impact of stochastic disturbances on the transient stability. Concepts like stochastic stability, stability of moments, and almost sure stability will be used to classify the resulting stochastic transient dynamics. Finally, a third direction involves the extension to the case of a single or multiple heterogeneous populations of micro-grids. In this context we will try to gain a better insights on scalability properties and emergent behaviors. The latter is a terminology used in complex network theory to address macroscopic phenomena arising from microscopic behavioral patterns.

\end{document}